\documentclass[oneside,11pt,reqno]{amsart}

\usepackage{amssymb,amsmath}
\usepackage{amssymb,latexsym}
\usepackage{amsfonts,xcolor}
\usepackage{amssymb}
\usepackage{stmaryrd}
\usepackage{longtable}
\usepackage{amsmath, geometry, amssymb} 
\numberwithin{equation}{section}
\newtheorem{theorem}{Theorem}[section]
\newtheorem{lemma}[theorem]{Lemma}

\newtheorem{proposition}[theorem]{Proposition}

\def\bals#1\nals{\begin{align*}#1\end{align*}}
\def\bal#1\nal{\begin{align}#1\end{align}}

\newcommand{\pmd}{\hspace{-3mm} \pmod}

\newcommand{\qu}[2]{\Bigl({\frac{#1}{#2}}\Bigr)_{\mkern-6.2mu K} }
\newcommand{\dqu}[2]{\ds{\qu{#1}{#2}}}

\def \ds{\displaystyle}

\newcommand{\nn}{\mathbb{N}}
\newcommand{\zz}{\mathbb{Z}}
\newcommand{\qq}{\mathbb{Q}}
\newcommand{\cc}{\mathbb{C}}

\newcommand{\hh}{\mathbb{H}}

\newcommand{\Z}{\mathbb{Z}}
\newcommand{\cV}{\mathcal{V}}
\newcommand{\cU}{\mathcal{U}}
\newcommand{\N}{\mathbb{N}}
\newcommand{\Q}{\mathbb{Q}}
\DeclareMathOperator{\Gal}{Gal}

\begin{document}

\title{N-colored generalized Frobenius partitions: Generalized Kolitsch identities}


\author{Zafer Selcuk Aygin}
\address{
Zafer Selcuk Aygin \\
Department of Mathematics and Statistics\\
University of Calgary\\
AB T2N 1N4, Canada
}
\email{selcukaygin@gmail.com}

\author{Khoa D.~Nguyen}
\address{
Khoa D.~Nguyen \\
Department of Mathematics and Statistics\\
University of Calgary\\
AB T2N 1N4, Canada
}
\email{dangkhoa.nguyen@ucalgary.ca}

\begin{abstract}
Let $N\geq 1$ be squarefree with $(N,6)=1$. Let $c\phi_N(n)$ denote the number of $N$-colored generalized Frobenius partition of $n$ introduced by Andrews in 1984. We prove
$$
c\phi_N(n)= \sum_{d \mid N} N/d \cdot P\left( \frac{ N}{d^2}n  - \frac{N^2-d^2}{24d^2} \right) + b(n)$$
where
$C(z) := (q;q)^N_\infty\sum_{n=1}^{\infty} b(n) q^n$
is a cusp form in $S_{(N-1)/2} (\Gamma_0(N),\chi_N)$. This extends and strengthens earlier results of
Kolitsch and Chan-Wang-Yan treating the case when $N$ is a prime. As an immediate application, we obtain an asymptotic formula for $c\phi_N(n)$ in terms of the classical partition function. 
\end{abstract}

\maketitle



\section{Introduction}

Let $\nn$, $\nn_0$, $\zz$, $\qq$, $\cc$ and $\hh$ denote the sets of positive integers, non-negative integers, integers, rational numbers, complex numbers and upper half plane of complex numbers, respectively. Throughout the paper we denote $q=e^{2 \pi i z}$, where $z \in \hh$.

In 1984, Andrews \cite{andrews} introduced the function
$c\phi_N(n)$
counting the number of $N$-colored generalized Frobenius partition of $n$ with $N\in\N$ and $n\in\N_0$. The generating function of $c\phi_N(n)$ is denoted by
\bals
C\Phi_N(q):= \sum_{n = 0}^{\infty} c\phi_N(n)q^n.
\nals
Andrews \cite{andrews} determined $C\Phi_N(q)$ in terms of a theta function divided by an infinite product, as follows. Let
\bals
\theta_{N} (x) :=  \sum_{i = 1 }^{N} x_i^2 + \sum_{1 \leq i < j \leq N} x_i x_j.
\nals
be a quadratic form in $N$ variables, and
\bals
f_{\theta_{N}}(z):= \sum_{ x \in \zz^{N} } q^{\theta_{N}(x)},
\nals
be the associated theta function. Then by \cite[Theorem~5.2]{andrews}, we have
\bals
C\Phi_N(z) = \frac{f_{\theta_{N-1}}(z)}{(q;q)_{\infty}^N},
\nals
where
\bals
(q;q)_{\infty}=\prod_{n \geq 1} (1-q^n).
\nals

There has been a plethora of research concerning the congruence properties of $c\phi_N(n)$, we leave the discussion of this topic and related results to \cite{wangpaper} and its references. In this paper we shall investigate relations between $c\phi_N(n)$ and $P(n)$, where $P(n)$ denotes the number of partitions of $n$. 
From the description of $c\phi_N(n)$ in \cite{andrews} or \cite{wangpaper}, it is trivial to deduce
\bals
c\phi_1(n)=P(n).
\nals
In \cite{kolitsch,kol2}, Kolitsch has shown rather surprising relationships between these two types of partitions which are stated below.
\begin{theorem}[Kolitsch 1991] \label{kolth}
For all $n \in \nn_0$ we have
\bal
c\phi_5(n)= 5P(5n-1) + P(n/5), \label{kol1}
\nal
\bal
c\phi_7(n)= 7P(7n-2) + P(n/7), \label{kol2}
\nal
and
\bal
c\phi_{11}(n)= 11P(11n-5) + P(n/11). \label{kol3} 
\nal
\end{theorem}
The proof of these beautiful identities relies on a $q$-series identities from \cite[eqs.(2.2) and (3.1)]{tcoregarvan} that relates the generating function of t-cores to theta series.

Very recently in \cite{wangpaper}, Chan, Wang and Yan have discovered the following more general relationships between $c\phi_p(n)$ and $P(n)$. Below, noting that the Dedekind eta function is defined by $\eta(z)=q^{1/24} (q;q)_{\infty}$, we restate main aspects of their Theorem 4.1.
\begin{theorem}[Chan-Wang-Yan 2019] \label{wangth} For all $n \in \nn_0$, we have
\bal
c\phi_{13}(n)= 13P(13n-7) + P(n/13) + a(n), \label{kol4}
\nal
where 
$\displaystyle q \frac{(q^{13};q^{13})_\infty}{(q;q)^2_\infty} = \sum_{n=1}^{\infty} a(n) q^n$.
When $p \geq 17$ is a prime then we have 
\bals
\sum_{n = 0}^\infty \left(c\phi_{p}(n) - p \cdot P\left(pn - \frac{p^2-1}{24}\right) - P(n/p)\right)q^n =
\frac{h_p(z) + 2 p^{(p-11)/2}  (\eta(pz)/\eta(z))^{p-11}}{(q^p;q^p)_{\infty}},
\nals
where $h_p(z)$ is a modular function on $\Gamma_0(p)$ with a zero at $\infty$ and a pole of order $(p + 1)(p - 13)/24$ at $0$. Additionally, the function
$\displaystyle h_p(z) (\eta(z)\eta(pz))^{p-13}
$
is a holomorphic modular form of weight $p - 13$ with a zero of order
$(p - 1)(p - 11)/24$ at $\infty$ and $h_p(z)$ is congruent to $p$ times a cusp form on $\Gamma_0(1)$ of weight $p-1$ modulo $p^2$.
\end{theorem}

These results rely on some delicate residue calculations and properties of modular functions.
\emph{The goal of this paper} is to extend the above results of Kolitsch and Chan-Wang-Yan to give relations between $c\phi_N(n)$ and $P(n)$, where $N$ is a squarefree integer that is coprime to $6$. 
The methods we use is quite different than that of \cite{wangpaper} or \cite{kol2}. We describe our method after stating our main theorem.

We fix $\chi_a(b)$ to be the Kronecker symbol $\dqu{(-1)^{(a-1)/2}a}{b}$. Whenever $a$ is squarefree odd integer $\chi_a(b)$ is a modulo $a$ primitive Dirichlet character. The space of modular forms of weight $k$ for the modular subgroup $\Gamma_0(N)$ with multiplier system $\chi_N$ is denoted by  $M_{k} (\Gamma_0(N),\chi_N)$, and its subspace of cusp forms is denoted by  $S_{k} (\Gamma_0(N),\chi_N)$.

\begin{theorem}[Main Theorem] \label{mainth} Let $N$ be a squarefree positive integer with $(N,6)=1$.
\begin{itemize}
\item[i)] Then for all $n \in \nn_0$ we have
\bal
c\phi_N(n)= \sum_{d \mid N} N/d \cdot P\left( \frac{ N}{d^2}n  - \frac{N^2-d^2}{24d^2} \right) + b(n), \label{maineqn}
\nal
where
\bals
C(z) := (q;q)^N_\infty\sum_{n=1}^{\infty} b(n) q^n
\nals
is a cusp form in $S_{(N-1)/2} (\Gamma_0(N),\chi_N)$. 
\item[ii)] We have $C(z)=0$ if and only if $N=5,7,$ or $11$.
\item[iii)] If $N \neq 5,7,$ or $11$, then there is no $M\geq 0$ such that $b(n)=0$ for all $n >M$.
\end{itemize}
\end{theorem}

Theorem \ref{mainth} is a result of a chain of modular identities. We first discover an identity that relates the theta function $f_{\theta_{N-1}}(z)$ to Eisenstein series. Then we find another identity that relates these Eisenstein series to the partition function using an intimate relationship between eta quotients and Eisenstein series. These modular identities are determined by using \cite[Theorem 1.1]{projections}. Finally, we combine these identities to obtain Theorem \ref{mainth}.

In contrast with \cite[Theorem 4.1 (c)]{wangpaper}, when $N=p$ a prime greater than $13$, our theorem exposes slightly more about $h_p(z)$. As a result of our Theorem \ref{mainth} we obtain that
\bals
h_p(z) + 2 p^{(p-11)/2}  (\eta(pz)/\eta(z))^{p-11}
\nals
is simply a cusp form in $S_{(N-1)/2} (\Gamma_0(N),\chi_N)$. Therefore it is evident that $h_p(z) $ is congruent to a cusp form modulo $p^2$.

On the other hand when $N=p$ a prime greater than $3$ then our Theorem \ref{mainth} leads to the equation
\bals
c\phi_p(n)=  p \cdot P\left( p n  - \frac{p^2-1}{24} \right) + P\left( n/p \right)+ b(n),
\nals
where
\bals
C(z) := (q;q)^p_\infty\sum_{n=1}^{\infty} b(n) q^n
\nals
is a cusp form in $S_{(p-1)/2} (\Gamma_0(p),\chi_p)$. Therefore, \eqref{kol1}--\eqref{kol4} can easily be deduced from our Theorem \ref{mainth}. Using Sturm Theorem one observes that in the cases $N=5,7$ and $11$ we have $C(z)=0$, which leads to \eqref{kol1}--\eqref{kol3}.

As an application of Theorem~\ref{mainth}, we establish the following asymptotic formula for $c\phi_N(n)$
in terms of linear combinations of partition functions:

\begin{theorem} \label{lim_th}
Let $N$ be a positive integer with $(N,6)=1$. We have
$$
c\phi_N(n) \sim \sum_{d \mid N} N/d\cdot P \left( \frac{N}{d^2} n - \frac{N^2-d^2}{24 d^2}\right) 
$$
as $n\to\infty$.
\end{theorem}

The organization of the paper is as follows. In Section \ref{secnotation} we introduce further notation and prove an important theorem concerning the modular forms in $M_{k}(\Gamma_0(N),\chi_N)$, see Theorem \ref{spacegenth}. In Section \ref{sec2} we compute constant terms of $f_{\theta_{N-1}}(z)$ at the cusps $1/c$ where $c \mid N$. This requires computing some Gauss sums related to the quadratic form $\theta_{N-1}$. These Gauss sum computations could be of independent interest to an audience with particular interest in the subject. In Section \ref{sec3} we compute the constant terms of the eta quotient $\frac{\eta^N((N/d) z)}{\eta(dz)}$ at the cusps $1/c$ where $c \mid N$. In Section \ref{sec4} we use Theorem \ref{spacegenth} and the calculations of Sections \ref{sec2} and \ref{sec3} to give $f_{\theta_{N-1}}(z)$ and $\frac{\eta^N((N/d) z)}{\eta(dz)}$ in terms of Eisenstein series. We then use the relationship between $\frac{\eta^N((N/d) z)}{\eta(dz)}$ and the partition function to prove an identity relating Eisenstein series and the partition function. Then we combine these identities to prove Theorem \ref{mainth}. In Section \ref{sec5} we show that the error term $b(n)$ is much smaller than 
\bals
\sum_{d \mid N} N/d \cdot P\left( \frac{ N}{d^2}n  - \frac{N^2-d^2}{24d^2} \right),
\nals 
by combining estimates involving coefficients of various $q$-series and this proves Theorem~\ref{lim_th}.

\section{Notation and preliminaries} {\label{secnotation}

In this section we introduce further notation and prove a theorem on a certain space of modular forms, see Theorem \ref{spacegenth}. This theorem is the backbone of the paper. We start with some notation.

Recall that $\chi_a(b)$ denotes the Kronecker symbol $\dqu{(-1)^{(a-1)/2}a}{b}$. Let $k \in \nn$. The generalized sum of divisors function associated with $\chi_d$ and $\chi_{N/d}$ is defined by
\begin{align*}
& \sigma_{k-1}(\chi_{N/d},\chi_d; n) :=  \sum_{1 \leq t\mid n}\chi_{N/d}(n/t){\chi_d}(t)t^{k-1}.
\end{align*}

Let $B_{k,\chi_N}$ denote the $k$-th generalized Bernoulli number associated with $\chi_N$ defined by the series
\bals
\sum_{k = 0}^\infty \frac{B_{k,\chi_N}}{k !} t^k = \sum_{a=1}^{N} \frac{{\chi_N}(a) t e^{at}}{e^{N t}  -1}.
\nals

Let $a \in \zz$ and $c \in \nn_0$ be coprime. For an $f(z) \in M_{k}(\Gamma_0(N),\chi)$ we denote the constant term of $f(z)$ in the Fourier expansion of $f(z)$ at the cusp $a/c$ by
\begin{align*}
[f]_{a/c} = \lim_{z \rightarrow i \infty} (c z + d)^{-k} f \left( \frac{az+b}{cz+d} \right),
\end{align*}
where $b,d \in \zz$ such that $\begin{bmatrix} a & b \\ c & d \end{bmatrix} \in SL_2(\zz)$. The value of $[f]_{a/c} $ does not depend on the choice of $b, d$. Noting that throughout the paper we denote
\bals
\epsilon_c=\begin{cases}
1 & \mbox{ if $c \equiv 1 \pmd{4}$,}\\
i & \mbox{ if $c \equiv 3 \pmd{4}$,}
\end{cases}
\nals
we are ready to state and prove the following statement.

\begin{theorem} \label{spacegenth}
Let $f(z) \in M_{k}(\Gamma_0(N),\chi_N)$. Then we have
\bals
f(z) & = [f]_{1/N} + \sum_{d \mid N}  \frac{[f]_{1/d}}{ A(d,N) } \cdot \frac{(1-N) (N/d)^{(N-2)/2}}{B_{(N-1)/2,\chi_N}} \sum_{n \geq 1} \sigma_{(N-3)/2}(\chi_{N/d},\chi_{d};n) q^n \\
& + C(z),
\nals
where $C(z)$ is some cusp form in $S_{(N-1)/2} (\Gamma_0(N),\chi_N)$ and
\bals
A(d,N)=(-1)^{\frac{(d+1)(N/d-1)}{4}} \epsilon_{N/d}= \begin{cases}
1 & \mbox{ if $d \equiv 1 \pmod{4}$ and $N \equiv 1 \pmod{4}$,}\\
i & \mbox{ if $d \equiv 3 \pmod{4}$ and $N \equiv 1 \pmod{4}$,}\\
-i & \mbox{ if $d \equiv 1 \pmod{4}$ and $N \equiv 3 \pmod{4}$,}\\
1 & \mbox{ if $d \equiv 3 \pmod{4}$ and $N \equiv 3 \pmod{4}$.}
\end{cases}
\nals
\end{theorem}

\begin{proof} This theorem is a direct application of \cite[Theorem 1.1]{projections}. The specialized version of the set of tuples of characters defined in \cite{projections} and given below simplifies to
\bals
\mathcal{E}((N-1)/2,N,\chi_N):=& \{ (\epsilon,\psi) \in D(L,\cc) \times D(M,\cc) : \mbox{$\epsilon$, $\psi$ primitive, }  \\
& \quad \epsilon(-1) \psi(-1)=(-1)^{(N-1)/2},~ \epsilon \psi = \chi_N  \mbox{ and } LM \mid N \} \\
& = \{ (\chi_{N/d},\chi_d) : d \mid N \}.
\nals
Therefore by using \cite[Theorem 1.1]{projections} we obtain
\bal
f(z) & =  \sum_{d \mid N } \chi_{N/d}(-1) [f]_{1/d}\nonumber\\
 & \qquad \times \left(  \chi_{N/d}(0) + \frac{W(\chi_d)}{W(\chi_N)} \frac{(1-N)(N/d)^{(N-1)/2}}{B_{(N-1)/2,\chi_N}} \sum_{n \geq 1}  \sigma_{(N-3)/2}(\chi_{N/d},\chi_{d};n) q^n \right) \nonumber  \\
 & + C(z), \label{eqeis_1}
\nal
for some $C(z)$ in $S_{(N-1)/2} (\Gamma_0(N),\chi_N)$, where the Gauss sum $W(\chi_d)$ is defined by
\bals
W(\chi_d):=\sum_{a=1}^d \chi_d(a) e^{2 \pi i a/d}.
\nals
On the other hand since $N,d$ are squarefree and odd we have
\bals
\chi_N= \prod_{p \mid N} \chi_p, \mbox{ and }\chi_d= \prod_{p \mid d} \chi_p.
\nals
Additionally, we have
\bals
W(\chi_p) = \begin{cases}
\sqrt{p} & \mbox{ if $p \equiv 1 \pmd{4}$,}\\
i \sqrt{p} & \mbox{ if $p \equiv 3 \pmd{4}$.}
\end{cases}
\nals
By the multiplicative properties of the Gauss sums $W(\chi_N)$ for $p$ an odd prime divisor of $N$ we have
\bals
W(\chi_N)=(-1)^{(p-1)(N/p-1)/4} W(\chi_p) W(\chi_{N/p}),
\nals
see \cite[Lemma 3.1.2]{miyake}. Using this iteratively we deduce that
\bals
 W(\chi_N) =\epsilon_N \sqrt{N}  = \begin{cases}
\sqrt{N} & \mbox{ if $N \equiv 1 \pmod{4}$,}\\
i \sqrt{N}  & \mbox{ if $N \equiv 3 \pmod{4}$.}
\end{cases}
\nals
Putting this in \eqref{eqeis_1} we obtain the desired result.
\end{proof}

In order to get necessary modular identities from Theorem \ref{spacegenth} we need to compute $[f_{\theta_{N-1}}]_{1/d}$ and $\left[\frac{\eta^N((N/d) z)}{\eta(dz)}\right]_{1/d}$ for each $d \mid N$. Computation of $\left[\frac{\eta^N((N/d) z)}{\eta(dz)}\right]_{1/d}$ can be done using \cite[Proposition 2.1]{Kohler}. This is carried out in Section \ref{sec3}. By \cite[(10.2)]{wangpei} (see \cite[(1.9)]{projections} for a refined version) we have
\bal
[f_{\theta_{N-1}}]_{1/d}= \left( \frac{-i}{d} \right)^{(N-1)/2} \frac{G_{N-1}(1,d)}{\sqrt{N}}, \label{ftthetaeq}
\nal
where the quadratic Gauss sum $G_{N}(a,c)$ for $N,a,c \in \nn$ is defined by
\bals
G_N(a,c):=\sum_{\substack{x \in \zz^{N}\\ x \pmd{c}}} e^{2 \pi i a \theta_{N}(x)/c}.
\nals
Therefore to calculate $[f_{\theta_{N-1}}]_{1/d}$ we need to calculate $G_{N-1}(1,d)$, which is carried out in the next section.

\section{Gauss sums and constant terms of $f_{\theta_{N-1}}(z)$} \label{sec2}

Let $N$ be an odd squarefree positive integer. In this section we compute $G_{N-1}(a,d)$ for all $d \mid N$ and $a\in\N$ with $\gcd(a,d)=1$. Then when $\gcd(N,6)=1$, we use our computations together with \cite[(10.2)]{wangpei} to obtain the constant term $[f_{\theta_{N-1}}(z)]_{1/d}$ of $f_{\theta_{N-1}}(z)$ in its Fourier series expansion at $1/d$, see Theorem \ref{ctermtheta}. In this section, for a set $A$ and an $N$-tuple $x \in A^N$, we use the notation
$x = (x_1, \ldots, x_N)$, i.e. $x_i$ denotes the $i$-th coordinate of the tuple $x$. We first prove a multiplicativity result concerning $G_N(a,c)$. 

\begin{lemma} \label{lemmag1} Let $N \in \nn$. Let $\alpha,\beta, \gamma \in \nn$ be mutually coprime. Then we have
\bals
G_N(\gamma,\alpha \beta)=G_N(\beta \gamma, \alpha) G_N(\alpha \gamma,\beta).
\nals
\end{lemma}

\begin{proof}
The map $\Z/\alpha\Z\times \Z/\beta\Z\rightarrow \Z/\alpha\beta\Z$ given by $(x,y)\mapsto z=\beta x+\alpha y$ is bijective. Therefore, each $z\in (\Z/\alpha\beta\Z)^N$ can be expressed as 
$z=\beta x+\alpha y$ for unique $x\in(\Z/\alpha\Z)^N$, $y\in(\Z/\beta\Z)^N$. From
\bals
z_i= \beta \cdot x_i + \alpha \cdot  y_i,
\nals 
 we have
\bal
& z_i^2 \equiv (\alpha +\beta) (\beta \cdot x_i^2 + \alpha \cdot y_i^2 ) \pmd{a\beta}, \label{eqg1}\\
& z_i \cdot z_j \equiv (a+\beta) (\beta \cdot x_i x_j + \alpha \cdot y_i y_j ) \pmd{\alpha \beta}.\label{eqg2}
\nal
Using \eqref{eqg1} and \eqref{eqg2} we have
\bals
\sum_{i=1}^{N} z_i^2 + \sum_{\substack{i,j=1,\\ i<j}}^{N} z_iz_j & \equiv (\alpha+\beta) \left( \sum_{i=1}^{N} (\beta \cdot x_i^2 + \alpha \cdot y_i^2 ) + \sum_{\substack{i,j=1,\\ i<j}}^{N}  (\beta \cdot x_i x_j + \alpha \cdot y_i y_j ) \right)\\
& \equiv (\alpha+\beta) \left( \beta \sum_{i=1}^{N} x_i^2  + \beta \sum_{\substack{i,j=1,\\ j<k}}^{N} x_i x_j  + \alpha \sum_{i=1}^{N} y_i^2 + \alpha \sum_{\substack{i,j=1,\\ j<k}}^{N} y_i y_j \right).
\nals
Therefore, we have
\bals
G_N(\gamma,\alpha \beta)=& \sum_{z\in (\Z/\alpha\beta\Z)^N} e\left( \gamma \frac{\theta_{N}(z)}{\alpha \beta}  \right) = \sum_{z\in(\Z/\alpha\beta\Z)^N} e\left( \gamma  \frac{ \sum z_i^2 + \sum z_iz_j}{\alpha \beta}  \right) \\
& = \sum_{x\in(\Z/\alpha\Z)^N} e\left( \beta\gamma  \frac{ \sum x_i^2 + \sum x_ix_j}{\alpha} \right) \sum_{y \in (\Z/\beta\Z)^N}  e\left( \alpha \gamma  \frac{ \sum y_i^2 + \sum y_iy_j}{\beta}  \right)\\
&= \sum_{x\in(\Z/\alpha\Z)^N} e\left(\beta \gamma  \frac{ \theta_{N}(x)}{\alpha}  \right) \sum_{y\in (\Z/\beta\Z)^N} e\left( \alpha \gamma  \frac{ \theta_{N}(y)}{\beta}  \right)\\
& = G_N(\beta \gamma, \beta) G_N(\alpha \gamma,\beta).
\nals

\end{proof}

For an odd prime $p$, in order to relate the relevant quadratic Gauss sums (over $\Z/p\Z$) in $N$ variables to quadratic Gauss sums in $N-1$ or $N-2$ variables, we need the function 
$\mathcal{C}_p:\ (\Z/p\Z)\setminus \{1\bmod p\} \rightarrow \Z/p\Z$
given by 
$\mathcal{C}_p(R):=\displaystyle\frac{1}{4(1-R)}$. When $p$ is clear from the context, we can omit it from the subscript and use $\mathcal{C}(R)$ instead.

\begin{lemma} \label{lemmag2} Let $p$ be an odd prime. Let $N,R$ and $a$ be positive integers such that $\gcd(a,p)=1$. We have
\bals
& \sum_{x\in (\Z/p\Z)^N} e\left( a\frac{\theta_{N}(x)}{p} - a\frac{R x_{N}^2}{p}  \right) \\
& = \begin{cases}
\ds p  & \hspace{-2cm} \mbox{ if $R = 1 \pmd{p}$, and $N\leq 2$,}\\
\ds p \sum_{x\in (\Z/p\Z)^N} e\left( a \frac{ \theta_{N-2}(x)}{p} \right) = p \cdot G_{N-2}(a,p)  &  \hspace{-2cm} \mbox{ if $R \equiv 1 \pmd{p}$, and $N>2$,}\\
\epsilon_p \sqrt{p} \dqu{a(1-R)}{p} \sum_{x\in(\Z/p\Z)^{N-1}} e\left( a \frac{\theta_{N-1}(x)}{p} - a \frac{\mathcal{C}(R) x_{N-1}^2}{p}  \right) & \mbox{ if $R \not\equiv 1 \pmd{p}$.}
 \end{cases}
\nals
\end{lemma}
\begin{proof}
The following easy identities are used throughout the proof:
\bals
&\theta_{N}(x_1,\ldots, x_{N})=\theta_{N-1}(x_1,\ldots, x_{N-1}) + x_N \sum_{j=1}^N x_j,
\nals
\bals
\sum_{x\in\Z/p\Z}e\left(\frac{Ax^2+Bx+C}{p}\right)&=\sum_{y\in\Z/p\Z}e\left(\frac{Ay^2}{p}\right)e\left(\frac{C-(4A)^{-1}B^2}{p}\right)\\
&=e\left(\frac{C-(4A)^{-1}B^2}{p}\right)\dqu{A}{p}\epsilon_p \sqrt{p}
\nals

for $A,B,C\in\Z/p\Z$ with $A\neq 0$ 
by the change of variables $y=x+(2A)^{-1}B$. 

The case $R \equiv 1 \pmod{p}$ and $N=1$ is obvious:
\bals
\sum_{x_1=0}^{p-1} e\left( a \frac{\theta_1(x_1)-x_1^2}{p}  \right) =\sum_{x_1=0}^{p-1} 1=p.
\nals

When $R \equiv 1 \pmod{p}$ and $N=2$, we have
\bals
& \sum_{x\in (\Z/p\Z)^2}e\left( a \frac{\theta_2(x_1,x_2)-x_1^2}{p}  \right) =\sum_{x\in(\Z/p\Z)^2}e\left( \frac{ ax_2^2 + ax_1 x_2}{p}  \right)\\
& =\sum_{y_1,y_2\in\Z/p\Z}e\left(\frac{y_1y_2}{p}\right)\ \text{by changing variables $y_1=ax_2$ and $y_2=x_2+x_1$}\\
& =p\ \text{by orthogonality.}
\nals

Next we prove the case $R \equiv 1 \pmod{p}$ and $N > 2$. We have
\bal
& \sum_{x\in (\Z/p\Z)^N} e\left( a\frac{\theta_{N}(x)}{p} - a\frac{R x_{N}^2}{p}  \right) \nonumber \\
& = \sum_{x\in (\Z/p\Z)^N} e\left( a\frac{\theta_{N}(x)}{p} - a\frac{ x_{N}^2}{p} \right) \nonumber \\
& = \sum_{x\in (\Z/p\Z)^N} e\left( a\frac{\theta_{N-1}(x_1,\ldots, x_{N-1}) + x_N \sum_{j=1}^{N-1} x_j}{p} \right) \nonumber \\
& =\sum_{A=0}^{p-1}   \sum_{\substack{x \in (\Z/p\Z)^{N-1}\\ \sum x_i \equiv A\bmod p }} e\left( a\frac{\theta_{N-1}(x) }{p} \right) \sum_{x_N=0}^{p-1} e\left( \frac{ (aA) x_N }{p} \right) 
\label{eqgit_1}
\nal
Now we observe that if $A \not\equiv 0 \bmod p$ then $\displaystyle\sum_{x_N\in\Z/p\Z} e\left( \frac{ (aA) x_N }{p} \right) =0$. Therefore the RHS of \eqref{eqgit_1} is
\bal
p  \sum_{\substack{x \in(\Z/p\Z)^{N-1} \\ \sum x_i = 0 }} e\left( a\frac{\theta_{N-1}(x) }{c} \right). 
\nal
We use $-\sum_{j=1}^{N-2} x_j = x_{N-1}$ to eliminate $x_{N-1}$ so that the above expression is
\bals
p  \sum_{\substack{x \in (\Z/p\Z)^{N-1}\\ \sum x_i \equiv 0 }} e\left( a\frac{\theta_{N-2}(x_1,\ldots,x_{N-2}) }{p} \right) = p \sum_{x \in(\Z/p\Z)^{N-2}} e\left( a\frac{\theta_{N-2}(x) }{p} \right).
\nals

Finally we prove the case $R \not\equiv 1 \pmod{c}$.  We have
\bal
& \sum_{x\in(\Z/p\Z)^N} e\left( a\frac{\theta_{N}(x)}{p} - a\frac{R x_{N}^2}{p}  \right) \nonumber\\
& = \sum_{x\in(\Z/p\Z)^N} e\left( a\frac{\theta_{N-1}(x_1,\ldots,x_{N-1}) + (1-R)x_N^2 + x_N \sum_{j=1}^{N-1} x_j}{p} \right) \nonumber \\
& =  \sum_{A=0}^{p-1}   \sum_{\substack{x \in (\Z/p\Z)^{N-1}\\ \sum x_i \equiv A }} e\left( a\frac{\theta_{N-1}(x) }{p} \right) \sum_{x_N\in\Z/p\Z} e\left( \frac{ a (1-R) x_N^2 + (aA) x_N }{p} \right). \label{gausseq_1}
\nal
Now in \eqref{gausseq_1} we use 
\bals
\sum_{x_N\in\Z/p\Z} e\left( \frac{ a (1-R) x_N^2 + (aA) x_N }{p} \right)=\epsilon_p \sqrt{p} \dqu{a(1-R)}{p} e\left(- \frac{\mathcal{C}(R)a A^2}{p} \right)
\nals
so that the RHS of \eqref{gausseq_1} becomes
\bal
& \epsilon_p \sqrt{p} \dqu{a(1-R)}{p}  \sum_{A=0}^{p-1}  \sum_{\substack{x \in(\Z/p\Z)^{N-1} \\ \sum x_i \equiv A }} e\left( a\frac{\theta_{N-1}(x) - \mathcal{C}(R) A^2}{p} \right) \nonumber \\
& = \epsilon_p \sqrt{p} \dqu{a(1-R)}{p}  \sum_{A=0}^{p-1}  \sum_{\substack{x \in(\Z/p\Z)^{N-1} \\ \sum x_i \equiv A }} e\left( a\frac{\theta_{N-2}(x_1,\ldots,x_{N-2}) + x_{N-1} \sum_{j=1}^{N-1} x_j - \mathcal{C}(R) A^2}{p} \right). \label{gausseq_3}
\nal
We employ  $ x_{N-1} =A- \sum_{j=1}^{N-2} x_j  $ in \eqref{gausseq_3} so that its RHS is
\bal
& \epsilon_p \sqrt{p} \dqu{a(1-R)}{p}  \sum_{A=0}^{p-1}  \sum_{x\in (\Z/p\Z)^{N-2}} e\left( a\frac{\theta_{N-2}(x) + A ( A- \sum_{j=1}^{N-2} x_j) - \mathcal{C}(R) A^2}{p} \right) \nonumber \\
& = \epsilon_p \sqrt{p} \dqu{a(1-R)}{p}  \sum_{A=0}^{p-1}  \sum_{x\in (\Z/p\Z)^{N-2}} e\left( a\frac{\theta_{N-2}(x) + A^2 - A\sum_{j=1}^{N-2} x_j - \mathcal{C}(R) A^2}{p} \right). \label{gausseq_4}
\nal
Then we replace $ A$ by $ -A$ in \eqref{gausseq_4} to have 
\bal
& \epsilon_p \sqrt{p} \dqu{a(1-R)}{p}  \sum_{A=0}^{p-1}  \sum_{x\in (\Z/p\Z)^{N-2}} e\left( a\frac{\theta_{N-2}(x) + A^2 + A\sum_{j=1}^{N-2} x_j - \mathcal{C}(R) A^2}{p} \right) \nonumber \\
& = \epsilon_p \sqrt{p} \dqu{a(1-R)}{p} \sum_{x\in (\Z/p\Z)^{N-1}} e\left( a\frac{\theta_{N-1}(x) -\mathcal{C}(R) x_{N-1}^2}{p} \right) 
\nal
where $x\in (\Z/p\Z)^{N-1}$ in the last sum has the form
$x=(x_1,\ldots,x_{N-1},A)$ for an arbitrary $(x_1,\ldots,x_{N-1})\in(\Z/p\Z)^{N-1}$ and
$A\in\Z/p\Z$.
\end{proof}

We want to show that sufficiently many iterations of Lemma \ref{lemmag2} will relate $G_{N_1}(a,c)$ to $G_{N_2}(a,c)$ where $N_1>N_2$. For any positive integer $t$, let $\mathcal{C}^t$ denote the 
$t$-th fold iterate of $\mathcal{C}$. The value of $\mathcal{C}^t(R)$ is well-defined when none of the 
$R,\mathcal{C}(R),\ldots,\mathcal{C}^{t-1}(R)$
is $1\bmod p$. When $t=0$, we let $\mathcal{C}^t$ be the identity function on
$\Z/p\Z\setminus\{1\bmod p\}$. The next lemma describes the orbit of $0\bmod p$ under $\mathcal{C}$.
\begin{lemma} \label{lemmag3}
Let $p$ be an odd prime, we have the following:
\begin{itemize}
 \item [(i)] $\mathcal{C}^t((p+1)/2)=(p+1)/2$ for every $t\in\N$.
 
 \item [(ii)] $\{\mathcal{C}^t(0):\ t=0,1\ldots,p-2\}=\{0,1\ldots,(p-1)/2,(p+3)/2,\ldots,p-1\} \bmod p$
 with $\mathcal{C}^{p-2}(0)=1 \bmod p$.
 \end{itemize}
\end{lemma}
\begin{proof}
Part (a) follows from the fact that $\mathcal{C}((p+1)/2)=(p+1)/2$. For part (b), one can prove by induction on $t$ the formula:
$$\mathcal{C}^t(0)=\frac{t}{2t+2} \bmod p\ \text{for $0\leq t\leq p-2$}.$$
\end{proof}

\begin{proposition} \label{propg1}
Let $p$ be an odd prime, $N \in \nn$ be such that $N \geq p-1$ and $a \in \nn$ be coprime to $p$. Then we have 
\bals
G_N(a,p)= \begin{cases}
i^{(p-p^2)/2} \cdot \dqu{a}{p} p^{p/2}    & \mbox{ if $N=p-1$, or $p$,}\\
i^{(p-p^2)/2} \cdot \dqu{a}{p} p^{p/2} G_{N-p}(a,p) & \mbox{ if $N>p$.}
\end{cases}
\nals
\end{proposition}

\begin{proof} 
By Lemma \ref{lemmag3}, we have $\mathcal{C}^t(0) \not \equiv 1 \pmod{p}$ for $0 \leq t \leq p-3$. Therefore we apply Lemma \ref{lemmag2} repeatedly for $p-2$ many times and obtain
\bal
G_N(a,p)& = \left(\epsilon_p \sqrt{p} \dqu{a}{p} \right)^{p-2} \prod_{t=1}^{p-2} \dqu{1-\mathcal{C}^{t-1}(0)}{p} \nonumber \\
& \times  \sum_{x\in (\Z/p\Z)^{N-(p-2)}} e\left( a \frac{\theta_{N-(p-2)}(x)}{p} - a \frac{\mathcal{C}^{p-2}(0) x_{N-(p-2)}^2}{p}  \right) \nonumber \\
& = \left(\epsilon_p \sqrt{p} \dqu{a}{p} \right)^{p-2} \prod_{t=1}^{p-2} \dqu{1-\mathcal{C}^{t-1}(0)}{p} \nonumber \\
& \times  \sum_{x\in (\Z/p\Z)^{N-(p-2)}} e\left( a \frac{\theta_{N-(p-2)}(x)}{p} - a \frac{ x_{N-(p-2)}^2}{p}  \right), \label{eqg2_3}
\nal
where in the second step we use $\mathcal{C}^{p-2}(0) \equiv 1 \pmod{p}$ that comes from Lemma \ref{lemmag3}. When $N >p$, we apply Lemma \ref{lemmag2} to \eqref{eqg2_3} to obtain
\bals
G_N(a,p)& = \left(\epsilon_p \sqrt{p} \dqu{a}{p} \right)^{p-2} \prod_{t=1}^{p-2} \dqu{1-\mathcal{C}^{t-1}(0)}{p} \cdot p \cdot   \sum_{x\in (\Z/p\Z)^{N-p}} e\left( a \frac{\theta_{N-p}(x)}{p}  \right)\\
& = \left(\epsilon_p  \right)^{p-2} \dqu{a}{p} \prod_{t=0}^{p-2} \dqu{1-\mathcal{C}^{t-1}(0)}{p} \cdot p^{p/2} \cdot   G_{N-p}(a,p).
\nals
Finally the desired result follows by employing the elementary identities
\bal
\epsilon_p=i^{(1-p)/2} \dqu{(p+1)/2}{p} \label{elgeq_1}
\nal
and 
\bal
 \prod_{t=1}^{p-2} \dqu{1-\mathcal{C}^{t-1}(0)}{p}=(-1)^{(p-1)/2}\dqu{(p+1)/2}{p}. \label{elgeq_2}
\nal

When $N =p-1,$ or $p$, by similar arguments we obtain
\bals
G_N(a,p)& = \left(\epsilon_p \sqrt{p} \dqu{a}{p} \right)^{p-2} \prod_{t=1}^{p-2} \dqu{1-\mathcal{C}^{t-1}(0)}{p} \cdot p. 
\nals
The desired result in this case follows similarly by employing \eqref{elgeq_1} and \eqref{elgeq_2}.
\end{proof}

\begin{proposition} \label{propg2}
Let $N$ be an odd positive squarefree integer and let $p$ be a prime divisor of $N$. Then we have
\bals
G_{N-1}(a,p)= i^{(N-Np)/2} \cdot \dqu{a}{p} p^{N/2} .
\nals
\end{proposition}

\begin{proof} We apply Proposition \ref{propg1} to $G_{N-1}(a,p)$ for $N/p-1$ many times and obtain
\bals
G_{N-1}(a,p)= \left( i^{(p-p^2)/2} \cdot \dqu{a}{p} p^{p/2} \right)^{N/p} = i^{(N-Np)/2} \cdot \dqu{a}{p} p^{N/2}.
\nals
\end{proof}

\begin{theorem} \label{gaussres}
Let $N$ be an odd positive squarefree integer, let $d$ be a divisor of $N$, and let $a\in\Z$ with $\gcd(a,d)=1$. Then we have
\bals
G_{N-1}(a,d)= \dqu{a}{d}\cdot i^{(N-Nd)/2} \cdot d^{N/2}. 
\nals
\end{theorem}

\begin{proof} First, we compute $G_{N-1}(1,d)$. By Lemma \ref{lemmag1} and Proposition \ref{propg2} we have
\bals
G_{N-1}(1,d) & = \prod_{p \mid d} G_{N-1}(d/p,p) = \prod_{p \mid d}  i^{(N-Np)/2} \cdot \dqu{d/p}{p} p^{N/2}\\
& = d^{N/2} \prod_{p \mid d} i^{(N-Np)/2} \cdot \dqu{d/p}{p}.
\nals
We let
\bals
B(d,N):= \frac{\prod_{p \mid d} i^{(N-Np)/2} \cdot \dqu{d/p}{p}}{i^{(N-Nd)/2} }.
\nals
Now let $p_1$ be an odd prime such that $p_1 \nmid N$. Then for all $d \mid N$ we have
\bal
B(d,N p_1) & =  \frac{\prod_{p \mid d} i^{(Np_1-Np_1p)/2} \cdot \dqu{d/p}{p}}{i^{(Np_1-Np_1d)/2} } = (B(d,N))^{p_1}, \label{eqadn_1}
\nal
and
\bal
B(dp_1 ,N p_1) & =  \frac{\prod_{p \mid dp_1} i^{(Np_1-Np_1p)/2} \cdot \dqu{dp_2/p}{p}}{i^{(Np_1-Np_1d)/2} } \nonumber \\
& = \frac{i^{(Np_1-Np_1^2)/2} \cdot \dqu{d}{p_1} \prod_{p \mid d}  \dqu{p_1}{p} \prod_{p \mid d} i^{(Np_1-Np_1p)/2} \cdot \dqu{d/p}{p}}{i^{(Np_1-Ndp_1^2)/2} } \nonumber\\
& = \frac{ (-1)^{(p_1-1)(d-1)/4} \prod_{p \mid d} i^{(Np_1-Np_1p)/2} \cdot \dqu{d/p}{p}}{i^{(Np_1^2-Ndp_1^2)/2} } \nonumber\\
& = \frac{ (-1)^{(p_1-1)(d-1)/4}  (B(d,N))^{p_1}  }{i^{(Np_1^2-Ndp_1^2-Np_1+Ndp_1)/2} } \nonumber\\
& = \frac{ (-1)^{(p_1-1)(d-1)/4}  (B(d,N))^{p_1}  }{ (i^{(p_1-1)(1-d)/2})^{Np_1} } \nonumber\\
& = (B(d,N))^{p_1}. \label{eqadn_2}
\nal
Clearly $B(1,1)=1$. Therefore by \eqref{eqadn_1} and \eqref{eqadn_2} we have $B(d,N)=1$ and this proves
\bals
G_{N-1}(1,d)= i^{(N-Nd)/2} \cdot d^{N/2}.
\nals

We now compute $G_{N-1}(a,d)$. For $n\in\N$, let $\zeta_n:=\exp(2\pi i/n)$. Let $\sigma$ be the automorphism of $\Q(\zeta_d)$ such that $\sigma(\zeta_d)=\zeta_d^a$. This yields
\begin{equation}\label{eq:G(a,d) as sigma(G(1,d))}
G_{N-1}(a,d)=\sigma(G_{N-1}(1,d)).
\end{equation}
Let $k$ be the number of prime divisors of $d$ that are congruent to $1$ mod $4$. From
$$\prod_{p\mid d}(\epsilon_p\sqrt{p})^N=i^{Nk} d^{N/2}$$
and the fact that $k$ and $(1-d)/2$ have the same parity, we have:
\begin{equation}\label{eq:G(1,d) as Nth power of prod epsilonp sqrtp}
G_{N-1}(1,d)=\pm \prod_{p\mid d}(\epsilon_p\sqrt{p})^N.
\end{equation}

For each prime $p\mid d$, the field $\Q(\epsilon_p\sqrt{p})$ is the unique quadratic subfield of
$\Q(\zeta_p)$ which is also the fixed field of the quadratic residues in
$(\Z/p\Z)^*\cong \Gal(\Q(\zeta_p)/\Q)$. Therefore the restriction of $\sigma$ on $\Q(\epsilon_p\sqrt{p})$
maps
$$\epsilon_p\sqrt{p}\mapsto \dqu{a}{p}\epsilon_p\sqrt{p}.$$
Together with \eqref{eq:G(a,d) as sigma(G(1,d))} and \eqref{eq:G(1,d) as Nth power of prod epsilonp sqrtp}, we have:
$$G_{N-1}(a,d)=\prod_{p\mid d}\dqu{a}{p} G_{N-1}(1,d)=\dqu{a}{d} G_{N-1}(1,d)$$
and this finishes the proof.
\end{proof}

\begin{theorem} \label{ctermtheta}
Let $N$ be a positive squarefree integer such that $\gcd(N,6)=1$ and $d$ be a divisor of $N$. Then we have
\bals
[f_{\theta_{N-1}}(z)]_{1/d}=i^{(1-Nd)/2} \cdot \sqrt{d/N}.
\nals
\end{theorem}

\begin{proof}
We put the result of Theorem \ref{gaussres} in \eqref{ftthetaeq} to obtain the desired result.
\end{proof}

\section{Constant terms of $\frac{\eta^N((N/d) z)}{\eta(dz)}$} \label{sec3}

Throughout this section we let $N$ to be a positive squarefree integer such that $\gcd(N,6)=1$. We denote by $V_{1/c} \left( \frac{\eta^N((N/d) z)}{\eta(dz)} \right)$ the order of vanishing of the eta quotient $\frac{\eta^N((N/d) z)}{\eta(dz)} $ at the cusp $1/c$. We first show that $\frac{\eta^N((N/d) z)}{\eta(dz)} $ vanishes at all $1/c$ except when $c=d$.

\begin{lemma} \label{lemsec3_1} We have $V_{1/c} \left( \frac{\eta^N((N/d) z)}{\eta(dz)} \right)=0$ if $c=d$ and 
$V_{1/c} \left( \frac{\eta^N((N/d) z)}{\eta(dz)} \right)>0$ otherwise.
\end{lemma}
\begin{proof} By \cite[Proposition 5.9.3]{cohenbook} (with cusp width $N/c$) we have
\bals
V_{1/c} \left( \frac{\eta^N((N/d) z)}{\eta(dz)} \right) = \frac{N}{24 c} \left(\frac{d^2\gcd(N/d,c)^2-\gcd(d,c)^2}{d} \right).
\nals
Since $N$ is squarefree $\gcd(N/d,d)=1$, we have
\bals
V_{1/d} \left( \frac{\eta^N((N/d) z)}{\eta(dz)} \right)= \frac{N}{24 d} \left(\frac{d^2\gcd(N/d,d)^2-\gcd(d,d)^2}{d} \right) =0.
\nals 
If $c \neq d$ then we have $d>\gcd(d,c)$ and clearly $\gcd(N/d,c) \geq 1$, therefore $d^2\gcd(N/d,c)^2>\gcd(d,c)^2$. Hence we have
\bals
V_{1/c} \left( \frac{\eta^N((N/d) z)}{\eta(dz)} \right) = \frac{N}{24 c} \left(\frac{d^2\gcd(N/d,c)^2-\gcd(d,c)^2}{d} \right) > 0.
\nals
\end{proof}

Now we compute $\left[ \frac{\eta^N((N/d) z)}{\eta(dz)} \right]_{1/c}$ for all $c \mid N$. Note that below we use the notation $e(x) := e^{2 \pi i x} $. 
\begin{lemma} \label{lemfteta} Let $c\mid N$. Then we have
\bals
\left[ \frac{\eta^N((N/d) z)}{\eta(dz)} \right]_{1/c} = \begin{cases}
\dqu{N/d}{d} e\left( \frac{1}{8} ( 1- Nd ) \right) \cdot \left(\frac{d}{N}\right)^{N/2} & \mbox{ if $c=d$,}\\
0 & \mbox{ otherwise.} 
\end{cases}
\nals
\end{lemma}

\begin{proof} The case where $c \neq d$ is a direct result of Lemma \ref{lemsec3_1}. Now we prove the case when $c=d$.
Let $L_1= \begin{bmatrix} 1 & 0 \\ 1 & 1 \end{bmatrix}$ and $L_2= \begin{bmatrix} N/d & v \\ d & w \end{bmatrix} \in SL_2(\zz)$. Then by \cite[Proposition 2.1]{Kohler} we have
\bals
\left[ \frac{\eta^N((N/d) z)}{\eta(dz)} \right]_{1/d} & = \frac{\nu^N(L_2) e(- d v/24)}{\nu(L_1)} \left(\frac{d}{N}\right)^{N/2},
\nals
where
\bals
& \nu(L_1)= e\left( \frac{-1}{24} \right) ,\\
& \nu(L_2)= \dqu{w}{d} e\left( \frac{1}{24} ((N/d+w)d - vw(d^2-1) -3 d ) \right).
\nals
Then we have
\bals
\frac{\nu^N(L_2)e(- d v/24)}{\nu(L_1)} & = \dqu{w}{d}^N e\left( \frac{1}{24} (N(N/d+w)d - Nvw(d^2-1) -3N d+1 - d v) \right)\\
& = \dqu{w}{d}^N e\left( \frac{1}{24} ( vd-3Nd+3 - d v) \right)\\
& = \dqu{N/d}{d} e\left( \frac{1}{8} ( 1- Nd ) \right),
\nals
where in the first step we use $d^2-1 \equiv 0 \pmod{24}$, $N^2 \equiv 1 \pmod{24}$ and $Ndw \equiv 1+ dv \pmod{24}$, in the last step we use $N$ is an odd integer, and $w \cdot N/d \equiv 1 \pmod{d}$. 
\end{proof}



\section{Relations among $\frac{\eta^N((N/d) z)}{\eta(dz)}$, Eisenstein series, $P(n)$ and $f_{\theta_{N-1}}(z)$} \label{sec4}

The end goal of this section is to prove Theorem \ref{mainth}. We first prove a relationship between $\frac{\eta^N((N/d) z)}{\eta(dz)}$ and Eisenstein series, see Theorem \ref{etaitoeis}. Next we prove a relationship between Eisenstein series and the partition function, see Theorem \ref{etaitopart}. To do this we uncover a relationship between $\frac{\eta^N((N/d) z)}{\eta(dz)}$ and partition function using arithmetic properties of Eisenstein series. We then prove another identity relating $f_{\theta_{N-1}}(z)$ to Eisenstein series, see Theorem \ref{th5_3}. Finally we show that Theorem \ref{mainth} is a result of combination of these relations.

Now we state and prove the relationship between $\frac{\eta^N((N/d) z)}{\eta(dz)}$ and Eisenstein series.
\begin{theorem} \label{etaitoeis} Let $N$ be a positive squarefree integer such that $\gcd(N,6)=1$. Then we have
\bals
\frac{\eta^N((N/d) z)}{\eta(dz)} =& \chi_{N/d}(0) + \dqu{N/d}{d}  C(d,N)\cdot \frac{d}{N} \cdot \frac{(1-N)}{B_{(N-1)/2,\chi_N}} \cdot \sum_{n \geq 1} \sigma_{(N-3)/2}(\chi_{N/d},\chi_{d};n) q^n \\
& + C_1(z),
\nals
where $C_1(z) \in S_{(N-1)/2} (\Gamma_0(N),\chi_N)$ and
\bals
C(d,N):= \frac{i^{(1-Nd)/2}}{A(d,N)}=\dqu{-8}{N}\dqu{8}{d}\dqu{-4}{d}^{(N-1)/2}.
\nals
\end{theorem}

\begin{proof} Let $N$ to be a positive squarefree integer such that $\gcd(N,6)=1$. Then by \cite[Propositions 5.9.2 and 5.9.3]{cohenbook} and Lemma \ref{lemsec3_1} we have
\bals
\frac{\eta^N((N/d) z)}{\eta(dz)} \in M_{(N-1)/2} (\Gamma_0(N),\chi_N).
\nals
Now the desired result follows by combining Theorem \ref{spacegenth} and Lemma \ref{lemfteta}. 

\end{proof}

Next we state and prove a relationship between $P(n)$ and Eisenstein series.
\begin{theorem} \label{etaitopart} Let $N$ be a positive squarefree integer such that $\gcd(N,6)=1$. Then we have
\bals
& \chi_{N/d}(0) +    C(d,N)\cdot (N/d)^{(N-3)/2} \frac{(1-N)}{B_{(N-1)/2,\chi_N}} \cdot \sum_{n \geq 1} \sigma_{(N-3)/2}(\chi_{N/d},\chi_{d};n) q^n\\
& =  N/d\cdot (q;q)^N_{\infty} \cdot \sum_{n \geq 0} P \left( \frac{N}{d^2} n - \frac{N^2-d^2}{24 d^2}  \right) q^n + C_2(z).
\nals
where $C_2(z)$ is some cusp form in $S_{(N-1)/2} (\Gamma_0(N),\chi_N)$.
\end{theorem}

\begin{proof} For $m \in \nn$ we define the operator $U(m)$ by
\bals
U(m) {\Big\vert} \sum_{n \geq 0} a_{n} q^n = \sum_{n \geq 0} a_{nm} q^n. 
\nals
Then we have
\bal
U(N/d) {\Big\vert}\frac{\eta^N((N/d) z)}{\eta(dz)}  = (q;q)^N_{\infty} \sum_{n \geq 0} P \left( \frac{N}{d^2} n - \frac{N^2-d^2}{24d^2}\right) q^n. \label{eqeta_1} 
\nal

On the other hand, we observe that
\bals
\sigma_{(N-3)/2}(\chi_{N/d},\chi_{d};n \cdot N/d ) = \chi_{d}(N/d) (N/d)^{(N-3)/2} \sigma_{(N-3)/2}(\chi_{N/d},\chi_{d}; n).
\nals
Therefore we have
\bal
& U(N/d) {\Big\vert} \left( \chi_{N/d}(0) + C(d,N) \dqu{N/d}{d}  \cdot \frac{d}{N} \cdot \frac{(1-N)}{B_{(N-1)/2,\chi_N}} \cdot \sum_{n \geq 1} \sigma_{(N-3)/2}(\chi_{N/d},\chi_{d};n) q^n \right)  \label{eqeta_2} \\
& = \chi_{N/d}(0) + \dqu{N/d}{d} \chi_{d}(N/d) (N/d)^{(N-3)/2} \frac{(1-N)d/N}{B_{(N-1)/2,\chi_N}} \cdot \sum_{n \geq 1} \sigma_{(N-3)/2}(\chi_{N/d},\chi_{d};n) q^n. \nonumber
\nal
Finally the result follows from combining \eqref{eqeta_1}, \eqref{eqeta_2}, Theorem \ref{etaitoeis}, the elementary equation
\bals
\dqu{N/d}{d} \chi_{d}(N/d)=1,
\nals 
and the property of modular forms that if $C_1(z) \in S_{(N-1)/2} (\Gamma_0(N),\chi_N)$ then
\bals
C_2(z) := U(N/d) {\Big\vert} C_1(z) \in S_{(N-1)/2} (\Gamma_0(N),\chi_N).
\nals
\end{proof}

\begin{theorem} \label{th5_3} We have
\bals
f_{\theta_{N-1}}(z) =& 1+ \sum_{d \mid N}  C(d,N) (N/d)^{(N-3)/2} \frac{(1-N)}{B_{(N-1)/2,\chi_N}} \cdot \sum_{n \geq 1} \sigma_{(N-3)/2}(\chi_{N/d},\chi_{d};n) q^n \\
& + C_3(z),
\nals
where $C_3(z)$ is some cusp form in $S_{(N-1)/2} (\Gamma_0(N),\chi_N)$.
\end{theorem}

\begin{proof} By \cite[Theorem 2.1]{wangpaper} we have $f_{\theta_{N-1}}(z) \in M_{(N-1)/2} (\Gamma_0(N),\chi_N)$. Therefore the result follows from combining Theorem \ref{spacegenth} and Theorem \ref{ctermtheta}.

\end{proof}

\begin{proof}[Proof of Theorem \ref{mainth}] We start by proving part i). By combining Theorem \ref{etaitopart} and Theorem \ref{th5_3} we obtain
\bal
f_{\theta_{N-1}}(z) & = (q;q)^N_{\infty} \sum_{d \mid N}   \frac{N}{d}  \sum_{n \geq 0} P \left( \frac{N}{d^2} n - \frac{N^2-d^2}{24 d^2}  \right) q^n  + C(z) \label{eq2_1}
\nal
for some $C(z) \in S_{(N-1)/2} (\Gamma_0(N),\chi_N)$. We divide both sides of \eqref{eq2_1} by $(q;q)^N_{\infty}$ to obtain
 \bal
\sum_{n \geq 0} c\phi_N(n) q^n & = \sum_{n \geq 0} \left( \sum_{d \mid N}  N/d \cdot P \left( \frac{N}{d^2} n - \frac{N^2-d^2}{24 d^2}  \right) \right)q^n  + \frac{C(z)}{(q;q)^N_{\infty}}. \label{eq2_2}
\nal
\eqref{maineqn} follows by comparing coefficients of $q^n$ in \eqref{eq2_2}.

Now we prove part ii) of Theorem \ref{mainth}. When $N\geq 29$ a squarefree positive integer coprime to $6$ and $d<N$ a divisor of $N$ then $ \frac{ N}{d^2}  - \frac{N^2-d^2}{24d^2} \leq 0$. Therefore by \eqref{maineqn} and $c\phi_N(1)=N^2$ we have
\bals
b(1)&  =c\phi_N(1) -\sum_{d \mid N} N/d \cdot P\left( \frac{ N}{d^2}  - \frac{N^2-d^2}{24d^2} \right) \\
& = c\phi_N(1) - P \left( \frac{1}{N} \right) =N^2 \neq 0. 
\nals
Hence when when $N \geq 29$ a squarefree positive integer coprime to $6$ we have $C(z) \neq 0$. Similarly when $N=13,17,19$, or $N=23$ by \eqref{maineqn} we have
\bals
b(1) & =c\phi_N(1) - N \cdot P\left( N  - \frac{N^2-1}{24} \right) - P\left( \frac{1}{N} \right) = \begin{cases} 
26 \neq 0 & \mbox{ if $N=13$,}\\ 
170 \neq 0 & \mbox{ if $N=17$,}\\
266 \neq 0 & \mbox{ if $N=19$,}\\
506 \neq 0 & \mbox{ if $N=23$.}
\end{cases} 
\nals
This shows that $C(z) \neq 0$ when $N=13,17,19$, or $N=23$. Therefore by \eqref{kol1}--\eqref{kol3} we have $C(z)=0$ if and only if $N=5,7,$ or $11$.

Finally we prove part iii) of the theorem. We prove it by contradiction. Assume that there was an $M \geq 0$ such that $b(n)=0$ for all $n>M$, then we would have 
\bals
\sum_{n =1}^M b_n q^n = \frac{C(z)}{(q;q)_\infty^N}.
\nals
Right hand side of this equation is a meromorphic modular function and left hand side is an exponential sum. This is possible only if $\frac{C(z)}{(q;q)_\infty^N}=0$, which is shown to be false unless $N=5,7,$ or $11$ in the proof of part ii) of the theorem.

\end{proof}

\section{Proof of Theorem \ref{lim_th}} \label{sec5}

Throughout this section let $N$ be a positive integer such that $\gcd(N,6)=1$. We will use the Vinogradov symbols and various asymptotic notations in estimates involving functions in $n$ where $n\in\N$ is large. The implicit constants in these estimates might depend on $N$ but they are independent of $n$.
 Let
\bals
\mathcal{U}(n):= \frac{1-N}{B_{(N-1)/2,\chi_N}} \sum_{d \mid N}  C(d,N) (N/d)^{(N-3)/2} \sigma_{(N-3)/2}(\chi_{N/d},\chi_{d};n).
\nals
We start by investigating the size of $\mathcal{U}(n)$.

\begin{lemma}
We have $U(n)>0$ for every $n\in\N$ and
\bals
\mathcal{U}(n) &\gg n^{(N-3)/2}\ \text{if $N>5$,}\\
\mathcal{U}(n) &\gg n/\log\log n\ \text{if $N=5$.}
\nals
\end{lemma}

\begin{proof} Let $n=\prod_{p \mid n} p^{e_p}$ be the prime factorization of $n$ and write $k=(N-3)/2$. Then we have
\bals
\sigma_{k}(\chi_{N/d},\chi_{d};n) & =\prod_{p \mid n} \frac{(\chi_{d}(p) p^k)^{e_p+1} - \chi_{N/d}(p)^{e_p+1}}{\chi_{d}(p) p^k - \chi_{N/d}(p)}\\
& = \prod_{\substack{p \mid n\\ p \mid d}}  \chi_{N/d}(p^{e_p})
 \prod_{\substack{p \mid n\\ p \mid N/d}} \chi_{d}(p^{e_p}) p^{ke_p}   \prod_{\substack{p \mid n\\ p \nmid N}} \frac{(\chi_{d}(p) p^k)^{e_p+1} - \chi_{N/d}(p)^{e_p+1}}{\chi_{d}(p) p^k - \chi_{N/d}(p)} \\
& = \prod_{\substack{p \mid n\\ p \mid d}}  \chi_{N/d}(p^{e_p})
 \prod_{\substack{p \mid n\\ p \nmid d}} \chi_{d}(p^{e_p})
 \prod_{\substack{p \mid n\\ p \mid N/d}} p^{ke_p} 
 \prod_{\substack{p \mid n\\ p \nmid N}} \frac{(p^k)^{e_p+1} - \chi_{N}(p)^{e_p+1}}{p^k - \chi_{N}(p)} .
\nals
Therefore by elementary manipulations we obtain
\bals
& \sum_{d \mid N}  C(d,N) (N/d)^{(N-3)/2}  \cdot \sigma_{(N-3)/2}(\chi_{N/d},\chi_{d};n)\\
& = \sum_{d \mid N} C(d,N)  (N/d)^{(N-3)/2}  \cdot \prod_{\substack{p \mid n\\ p \mid d}}  \chi_{N/d}(p^{e_p}) \prod_{\substack{p \mid n\\ p \nmid d}} \chi_{d}(p^{e_p}) \prod_{\substack{p \mid n\\ p \mid N/d}} p^{ke_p}  \\
& \qquad \times \prod_{\substack{p \mid n\\ p \nmid N}} \frac{(p^k)^{e_p+1} - \chi_{N}(p)^{e_p+1}}{p^k - \chi_{N}(p)} \\
& =  \dqu{-8}{N} N^{(N-3)/2} \prod_{\substack{p \mid n\\ p \mid N}} p^{ke_p}  \prod_{\substack{p \mid n\\ p \nmid N}} \frac{(p^k)^{e_p+1} - \chi_{N}(p)^{e_p+1}}{p^k - \chi_{N}(p)} 
\\
&\qquad \times \sum_{d \mid N} \dqu{-8}{N}C(d,N)  (1/d)^{(N-3)/2}  \cdot \prod_{\substack{p \mid n\\ p \nmid d}} \chi_{d}(p^{e_p}) \prod_{\substack{p \mid n\\ p \mid d}} \frac{ \chi_{N/d}(p^{e_p})}{p^{ke_p}} .
\nals
On the other hand 
\bals
\dqu{-8}{N}C(d,N) (1/d)^{(N-3)/2}  \cdot \prod_{\substack{p \mid n\\ p \nmid d}} \chi_{d}(p^{e_p}) \prod_{\substack{p \mid n\\ p \mid d}} \frac{ \chi_{N/d}(p^{e_p})}{p^{ke_p}}
\nals
is a multiplicative function of $d \mid N$. Therefore we have
{\footnotesize \bals
& \frac{1-N}{B_{(N-1)/2,\chi_N}} \sum_{d \mid N} C(d,N)  (N/d)^{(N-3)/2}  \cdot \sigma_{(N-3)/2}(\chi_{N/d},\chi_{d};n)\\
& =  \dqu{-8}{N} \frac{1-N}{B_{(N-1)/2,\chi_N}} N^{(N-3)/2} \prod_{\substack{p \mid n\\ p \mid N}} p^{ke_p}  \prod_{\substack{p \mid n\\ p \nmid N}} \frac{(p^k)^{e_p+1} - \chi_{N}(p)^{e_p+1}}{p^k - \chi_{N}(p)}  \\
&\qquad \times 
\prod_{\substack{s \mid N \\ s~prime}}   \left(1+ \dqu{-8}{N}C(s,N) (1/s)^{(N-3)/2}  \cdot \prod_{\substack{p \mid n\\ p \nmid s}} \chi_{s}(p^{e_p}) \prod_{\substack{p \mid n\\ p \mid s}} \frac{ \chi_{N/s}(p^{e_p})}{p^{ke_p}} \right)\\
& =  \dqu{-8}{N} \frac{1-N}{B_{(N-1)/2,\chi_N}} N^{(N-3)/2} n^{k} \prod_{\substack{p \mid n\\ p \nmid N}} \frac{ 1 - \chi_{N}(p)^{e_p+1}/(p^k)^{e_p+1}}{1 - \chi_{N}(p)/p^k} \\
&\qquad \times \prod_{\substack{s \mid N \\ s~prime}}   \left(1+ \dqu{-8}{N}C(s,N)(1/s)^{(N-3)/2}  \cdot \prod_{\substack{p \mid n\\ p \nmid s}} \chi_{s}(p^{e_p}) \prod_{\substack{p \mid n\\ p \mid s}} \frac{ \chi_{N/s}(p^{e_p})}{p^{ke_p}} \right).
\nals}
The product over primes $s\mid N$ is at least
$$\prod _{s\mid N}\left (1-\frac{1}{s}\right)$$
while the first product
\begin{align*}
\prod_{\substack{p \mid n\\ p \nmid N}} \frac{ 1 - \chi_{N}(p)^{e_p+1}/(p^k)^{e_p+1}}{1 - \chi_{N}(p)/p^k}&=\prod_{\substack{p \mid n\\ p \nmid N}}\left(1+\frac{\chi_N(p)p^{-k}-\chi_N(p)^{e_p+1}p^{-k(e_p+1)}}{1-\chi_N(p)p^{-k}}\right)\\
&\geq\prod_{\substack{p \mid n\\ p \geq 3}}\left(1-\frac{p^{-k}+p^{-2k}}{1-p^{-k}}\right).
\end{align*}
When $N>5$ and hence $k>1$, we have:
$$\prod_{\substack{p \mid n\\ p \geq 3}}\left(1-\frac{p^{-k}+p^{-2k}}{1-p^{-k}}\right)>\prod_{p\geq 3} \left(1-\frac{p^{-k}+p^{-2k}}{1-p^{-k}}\right)$$
which converges to a positive number. When $N=5$, we have:
$$\prod_{\substack{p \mid n\\ p \geq 3}}\left(1-\frac{p^{-1}+p^{-2}}{1-p^{-1}}\right)\gg\prod_{p\mid n} \left(1-\frac{1}{p}\right)=\frac{\varphi(n)}{n}\gg \frac{1}{\log\log n}.$$

It remains to show $\ds \dqu{-8}{N}\frac{1-N}{B_{(N-1)/2,\chi_N}} N^{(N-3)/2} > 0$. We have the relation between Dirichlet L-functions and Bernoulli numbers \cite[Theorem 3.3.4]{miyake}:
\bals
B_{(N-1)/2,\chi_N}&=\frac{2k! N^k}{(-1)^{(N-3)/2} (2 \pi i)^{(N-1)/2} W(\chi_N)}  L((N-1)/2,\chi_N)\\
&=\frac{1}{(-1)^{(N-3)/2} i^{(N-1)/2} \epsilon_N }  \frac{2k! N^{k-1/2} L((N-1)/2,\chi_N)}{ (2 \pi )^{(N-1)/2}}.
\nals
We have $\displaystyle\frac{2k! N^{k-1/2} L((N-1)/2,\chi_N)}{ (2 \pi)^{(N-1)/2}}>0$ and it is easy to check $\displaystyle\frac{1}{(-1)^{(N-3)/2} i^{(N-1)/2} \epsilon_N }= -\dqu{-8}{N}$. This completes the proof.
\end{proof}

For each non-negative integer $r$, we define $\cV_r(n)$ for $n\geq 0$ by:
$$\sum_{n \geq 0} \mathcal{V}_r(n) q^n = \frac{1}{(q;q)^r_{\infty}} = \left(\sum_{n \geq 0} P(n) q^n \right)^r = \sum_{n \geq 0} \sum_{\substack{x \in \nn_0^r\\ \sum x_i=n}} \prod_{i=1}^r P(x_i) q^n.$$
We have:
\begin{proposition}\label{prop:quotient of the V_r's}
For $r\geq 1$:
\begin{itemize}
	\item [(i)] $\displaystyle\lim_{n\to\infty}\frac{\cV_r(n)}{\cV_r(n-1)}=1$.
	
	\item [(ii)] $\displaystyle\lim_{n\to\infty}\frac{\cV_{r-1}(n)}{\cV_r(n)}=0$.
\end{itemize}
\end{proposition}
\begin{proof}
We use induction on $r$. When $r=1$, both (i) and (ii) hold since $\cV_1(n)=P(n)$ and $\cV_0(n)=0$ for $n>0$. Consider $r\geq 2$ and assume that both (i) and (ii) hold for $r-1$.

First, we prove part (ii) for $r$. We have:
$$\cV_r(n)=\cV_{r-1}(n)+\cV_{r-1}(n-1)P(1)+\cV_{r-1}(n-2)P(2)+\ldots+\cV_{r-1}(0)P(n).$$
By part (i) for $r-1$, for each fixed positive integer $k$, we have: 
$$\lim_{n\to\infty}\frac{\cV_{r-1}(n)}{\cV_{r-1}(n-k)}=1.$$
Therefore part (ii) for $r$ holds.

Finally we prove part (i) for $r$. It suffices to show that for any given $\epsilon>0$, we have
$$\frac{\cV_r(n)}{\cV_r(n-1)}<1+\epsilon\ \text{for all sufficiently large $n$.}$$
Fix $k$ such that $P(m)/P(m-1)<1+\epsilon/2$ for every $m\geq k$. Let 
$$S=\{x\in \N_0^r:\ \sum x_i=n\ \text{and}\ x_1\geq k\},$$
$$S'=\{x\in \N_0^r:\ \sum x_i=n\ \text{and}\ x_1< k\},$$

For each $x=(x_1,\ldots,x_r)\in S$, put $y=(y_1,\ldots,y_r)$ with $y_1=x_1-1$ and $y_i=x_i$ for $i\geq 2$. Then we have
$$\prod_{i=1}^r P(x_i)\Big/\prod_{i=1}^r P(y_i)=P(x_1)/P(x_1-1)<1+\epsilon/2$$
which implies
\begin{equation}\label{eq:sum over S less than 1+epsilon/2}
\left(\sum_{x\in S}\prod_{i=1}^r P(x_i)\right)\Big/\cV_{r}(n-1) < 1+\epsilon/2.
\end{equation}
On the other hand, we have 
$$\sum_{x\in S'}\prod_{i=1}^r P(x_i)=\sum_{j=0}^{k-1} P(j) \cV_{r-1}(n-j).$$
And since each $\cV_{r-1}(n-j)/\cV_{r}(n)\to 0$ as $n\to\infty$, we have:
$$\left(\sum_{x\in S'}\prod_{i=1}^r P(x_i)\right)\Big/\cV_{r}(n-1) < \epsilon/2$$
for all sufficiently large $n$. Combining this with \eqref{eq:sum over S less than 1+epsilon/2}, we finish the proof that $\cV_r(n)/\cV_r(n-1)<1+\epsilon$ for all sufficiently large $n$.
\end{proof}

\begin{proof}[Proof of Theorem \ref{lim_th}]
When $N=5$, $7$ or $11$ from \eqref{eq2_1} and Sturm theorem we have $c\phi_N(n)=\sum_{d \mid N} N/d \cdot P\left(\frac{N}{d^2} n - \frac{N^2-d^2}{24d^2} \right) (\neq 0)$. Therefore the statement for $N=5$, $7$ or $11$ follows immediately. From now on assume $N>11$.
By Theorem \ref{th5_3} we have
\bals
f_{\theta_{N-1}}(z) -1 - \sum_{n \geq 1} \mathcal{U}(n) q^n 
\nals
is a cusp form in $S_{(N-1)/2} (\Gamma_0(N),\chi_N)$. Therefore by \cite[Theorem 9.2.1.(a)]{cohenbook} we have
\bals
f_{\theta_{N-1}}(z) -1 -\sum_{n \geq 1}  \mathcal{U}(n) q^n = \sum_{n \geq 1} O(n^{(N-1)/4}) q^n.
\nals
On the other hand by Theorem \ref{etaitopart} we have
\bals
(q;q)^N_{\infty} \sum_{d \mid N}   (N/d)  \sum_{n \geq 0} P \left( \frac{N}{d^2} n - \frac{N^2-d^2}{24 d^2}  \right) q^n -1 - \sum_{n \geq 1} \mathcal{U}(n) q^n 
\nals
is also a cusp form in $S_{(N-1)/2} (\Gamma_0(N),\chi_N)$. Therefore by \cite[Theorem 9.2.1.(a)]{cohenbook} we have
\bals
& (q;q)^N_{\infty} \sum_{d \mid N}    (N/d)  \sum_{n \geq 0} P \left( \frac{N}{d^2} n - \frac{N^2-d^2}{24 d^2}  \right) q^n -1 - \sum_{n \geq 1} \mathcal{U}(n) q^n \\
& = \sum_{n \geq 1} O(n^{(N-1)/4}) q^n.
\nals

Now we let $\mathcal{V}(n):=\mathcal{V}_N(n)$ so that
\bals
\frac{1}{(q;q)^N_{\infty}} = \sum_{n \geq 0} \mathcal{V}(n) q^n.
\nals
With this notation and the earlier arguments, we obtain
\bal
c\phi_N(n) - \sum_{\l+m=n} \mathcal{V}(m) \mathcal{U}(\l) =  O\left(  \sum_{\l+m=n} \mathcal{V}(m)  \l^{(N-1)/4}\right),\label{eqasym_2}
\nal
and
\bal
& \sum_{d \mid N}    (N/d)  \sum_{n \geq 0} P \left( \frac{N}{d^2} n - \frac{N^2-d^2}{24 d^2}  \right) - \sum_{\l+m=n} \mathcal{V}(m) \mathcal{U}(\l) \nonumber \\
& \qquad =  O\left(  \sum_{\l+m=n} \mathcal{V}(m)  \l^{(N-1)/4}\right). \label{eqasym_1}
\nal

From \eqref{eqasym_2} and \eqref{eqasym_1} we have
\bals
& \lim_{n \rightarrow \infty } \frac{c\phi(n)}{ \ds \sum_{d \mid N}  (N/d) P \left( \frac{N}{d^2} n - \frac{N^2-d^2}{24 d^2} \right)}\\
& =  \lim_{n \rightarrow \infty } \frac{\ds \sum_{\l+m=n} \mathcal{V}(m) \mathcal{U}(\l) + O\left(  \sum_{\l+m=n} \mathcal{V}(m)  \l^{(N-1)/4}\right)}{\ds \sum_{\l+m=n} \mathcal{V}(m) \mathcal{U}(\l) + O\left(  \sum_{\l+m=n} \mathcal{V}(m)  \l^{(N-1)/4}\right)}.
\nals

To obtain the desired result, we prove:
\begin{equation}\label{eq:little o result}
\sum_{\ell+m=n}\cV(m)\ell^{(N-1)/4}=o\left(\sum_{\ell+m=n}\cV(m)\cU(\ell)\right)\ \text{as $n\to\infty$.}
\end{equation}

Let $\epsilon>0$. Since $N>11$, we have that $\cU(\ell)\gg \ell^{(N-3)/2}$ dominates $\ell^{(N-1)/4}$ when $\ell$ is large.  Choose $L_{\epsilon}>0$ such that for every $\ell\geq L_{\epsilon}$, we have
$\ell^{(N-1)/4}<\epsilon \cU(\ell)$. This yields:
\begin{equation}\label{eq:sum with ell geq L_epsilon}
 \sum_{\ell+m=n,\ell\geq L_{\epsilon}} \mathcal{V}(m)\ell^{(N-1)/4}<\epsilon\sum_{\ell+m=n}\mathcal{V}(m)\cU(\ell)
\end{equation}

Choose a positive integer $L'_\epsilon$ such that 
\begin{equation}\label{eq:choosing L'_epsilon}
\cU(\ell)>\epsilon^{-1} L_\epsilon^{(N+3)/4}\ \text{for every $\ell\geq L'_\epsilon$.}
\end{equation}

We now consider $\ell<L_\epsilon$, we have:
$$\mathcal{V}(n-\ell)\ell^{(N-1)/4}\leq \mathcal{V}(n-\ell)L_{\epsilon}^{(N-1)/4}\leq \frac{\epsilon}{L_\epsilon} \mathcal{V}(n-\ell)\cU(L'_\epsilon \ell).$$
Proposition~\ref{prop:quotient of the V_r's} implies that $\mathcal{V}(n-\ell)<2\mathcal{V}(n-L'_{\epsilon}\ell)$
for every $\ell<L_{\epsilon}$ and for every sufficiently large $n$. This yields
$$\mathcal{V}(n-\ell)\ell^{(N-1)/4}\leq \frac{2\epsilon}{L_\epsilon} \mathcal{V}(n-L'_\epsilon\ell)\cU(L'_\epsilon \ell)$$
and hence
\begin{equation}\label{eq:sum with ell<L_epsilon}
\sum_{\ell+m=n,\ell<L_{\epsilon}} \mathcal{V}(m)\ell^{(N-1)/4}<2\epsilon\sum_{\ell+m=n}\mathcal{V}(m)\cU(\ell)
\end{equation}
for all sufficiently large $n$. From \eqref{eq:sum with ell geq L_epsilon} and \eqref{eq:sum with ell<L_epsilon}, we have:
$$\sum_{\ell+m=n}\mathcal{V}(m)\ell^{(N-1)/4}<3\epsilon \sum_{\ell+m=n}\mathcal{V}(m)\cU(\ell)$$
for all sufficiently large $n$ and this finishes the proof.
\end{proof}

\section*{Acknowledgements}
Z~S.~Aygin is partially supported by a PIMS postdoctoral fellowship. Both Z.~S.~Aygin and K.~D.~Nguyen are partially supported by an NSERC Discovery Grant and a CRC tier-2 research stipend.

\end{document}